\documentclass[a4paper,11pt]{amsart}

\usepackage[english]{my-shortcuts}
\usepackage{srcltx,algorithm,algorithmic}

\begin{document}

	\title[Generic uniqueness for the LASSO]{On the uniform generic uniqueness of the LASSO estimator}
	\author{St\'ephane Chr\'etien and S\'ebastien Darses}

	\address{National Physical Laboratory\\ 
		Hampton Road,\\
		Teddington, TW11 0NS, UK} 
	\email{stephane.chretien@npl.co.uk}
	
	\address{LATP, UMR 6632\\
		Universit\'e de Provence, Technop\^ole Ch\^{a}teau-Gombert\\
		39 rue Joliot Curie\\ 13453 Marseille Cedex 13, France}
	\email{darses@cmi.univ-mrs.fr}

	\maketitle
	
	
	\begin{abstract}
		The LASSO is a variable subset selection procedure in statistical linear regression based 
		on $\ell_1$ penalization of the least-squares operator. Uniqueness of the LASSO 
		is an important issue, especially for the study of the LASSO path. 
		The goal of the present paper is to provide a generic sufficient condition on the 
		design matrix for the LASSO minimizer to be unique. Unlike previous 
		works on the question of uniqueness, our condition only depends on the 
		design matrix. 
		Our study is based on a general position condition on the design 
		matrix which holds with probability one for most experimental models. 
	\end{abstract}
	

	\section{Introduction}
	
	\subsection{Problem statement and main results}
	The well-known standard Gaussian linear model in statistics reads $y  =  X\beta+z$,
	where $X$ denotes a $n\times p$ design matrix, $\beta\in\R^{p}$ is an unknown parameter and the components of the error $z$ are assumed i.i.d. with normal distribution $\mathcal N(0,\sigma^2)$.  
	
	The case where $p$ is much larger than $n$ has been the subject of an intense recent study. This problem 
	is of course not solvable for any $\beta$ but it has been discovered that if $\beta$ is sufficiently sparse, then the solution of 
	\bea
	\label{lasso}
	\qquad \ch{\beta}_\lb & \in & \down{b\in \mathbb R^p}{\rm argmin} \ \frac12\|y-X b\|_2^2+\lb \|b\|_1,
	\eea
	called the LASSO estimator of $\beta$, is sometimes also sparse and close to $\beta$. The acronym LASSO, due to \cite{Tibshirani:JRSSB96}, stands for Least Absolute Shrinkage and Selection Operator, 
	and stems from the fact that the $\ell_1$-norm penalty shrinks the components of the standard least-squares 
	estimator $\ch{\beta}$. Some components are shrinked to the point of setting them to zero, hence implying 
	automatic selection of the remaining nonzero components as good predictors for the experiments under study. We refer the interested 
	reader to \cite{Candes:ActaNum06} and \cite{delBarrio:EMS07} for an overview of the relationships between sparsity and statistics, and 
	sparsity promoting penalizations of the least-squares criterion. Important  results concerning the LASSO and extensions to other statistical models and penalizations strategies may be found in \cite{Bunea:EJStat07}, \cite{Bickel:AnnStat09}, \cite{CandesPlan:AnnStat09}
	and \cite{vandeGeer:AnnStat10} for instance. We refer to 
	\cite{Hastie:SparsityBook15} for a thourough overview of the field.
	
	Under the assumption that the columns of $X$ are sufficiently "uncorrelated", several authors were able to prove that, 
	with high probability, the $\ell_2$-norm of $X(\ch{\beta}-\beta)$ is of the same order of magnitude as 
	the $\ell_2$-norm of $X(\tilde{\beta}-\beta)$ for an oracle $\tilde{\beta}$. It may even perform as well as an "oracle". For instance, the oracle proposed in \cite{CandesPlan:AnnStat09} is a solution of 
	\bean
	\label{racle}
	\qquad \tilde{\beta}_\lb & \in & \down{b\in \mathbb R^p, \ b_{T^c}=0}{\rm argmin} \ \frac12\|y-X_T b_T\|_2^2+\lb\ \sgn(\beta_T) ^tb_T,
	\eean
	where $T$ is the index set of the non-zero components of $\beta$.
	The term "oracle" is often used to emphasize that the support of $\beta$ is usually unknown ahead of time. 
	Under stronger assumptions it was further proven in \cite{Bunea:EJS08} and \cite{CandesPlan:AnnStat09} that the support and sign pattern of $\beta$ can be recovered exactly with high probability. The case where the variance 
	is unknown was treated in e.g. \cite{Baraud:StatScience12},
	\cite{ChretienDarses:IEEEIT14}. A very efficient algorithm, based on Nesterov's method, for 
	solving the LASSO estimation problem is described in \cite{Becker:SIAMImSc10}.
	
	One important property expected from the LASSO is the uniqueness of the 
	solution and continuity of the path (to say the least). Such properties
	used early in the study of the LASSO like for the LARS \cite{LARS},
	are fundamental for many interesting recent results about the LASSO; see e.g. 
	\cite{SuBogdanCandes:ArXiv15}. 
	
	The problem of establishing the uniqueness for the LASSO has been addressed in 
	a number of both early and recent results \cite{Fuchs:IEEEIT04},
	\cite{Osborne:LASSODual00}, \cite{Dossal:CRAS12}, \cite{Tibshirani:EJS13}, \cite{ZhangYinChen:JOTA15} and \cite{Gilbert:INRIAReport15}. The papers \cite{Fuchs:IEEEIT04}, \cite{Dossal:CRAS12} and \cite{Gilbert:INRIAReport15} mainly deal with the Basis Pursuit (i.e. the noisefree LASSO). The papers 
	\cite{Osborne:LASSODual00}, \cite{Tibshirani:EJS13} and \cite{ZhangYinChen:JOTA15} address uniqueness in 
	the LASSO but the conditions for uniqueness given in 
	these works depend on the design matrix $X$ and either the observation vector $y$ or the sign pattern of the original regression vector. 
	
	The goal of the present work is to provide a simple, uniform in $\lambda$, condition for uniqueness depending on the design matrix $X$ only. 
	
	\subsection{Plan of the paper}
	
	Section \ref{opt} recalls the optimality conditions associated to the LASSO. In Section \ref{lassolamb}, we study the 
	uniqueness of the standard LASSO estimator of $\beta$. 

	\subsection{Notations} 
	
	Let us briefly recall some basic notations. For $I\subset \left\{1,\ldots,p\right\}$, $|I|$ denotes the cardinal of $I$. For $x\in \R^{p}$, we set $x_{I}=(x_{i})_{i\in I}\in\R^{|I|}$.
	The usual scalar product is denoted by $\la\cdot,\cdot\ra$. 
	For any matrix $A$, we denote by $A^{t}$ its transpose. For $I\subset \left\{1,\ldots,p\right\}$, and a matrix $X$,
	we denote by $X_{I}$ the submatrix whose columns are indexed by $I$.
	
	The set of symmetric real matrices is denoted by 
	$\mathbb S_n$. For any matrix $A$ in $\mathbb R^{d_1\times d_2}$, we denote by $\|A\|$ the operator norm of $A$. 
	The maximum (resp. minimum) singular value of $A$ is denoted by $\sigma_{\max}$ (resp. $\sigma_{\min}(A)$). 
	Recall that $\sigma_{\max}(A)=\|A\|$ and $\sigma_{\min}(A)^{-1}=\|A^{-1}\|$.
	We use the Loewner ordering on symmetric real matrices: if $A\in \mathbb S_n$, 
	$0\preceq A$ is equivalent to saying that $A$ is positive semi-definite, and $A\preceq B$ stands for $0\preceq B-A$.
	
	For any vector $b\in\mathbb R^p$, $b^{+}$ (resp. $b^{-}$) denotes its non-negative (resp. non-positive)
	part, i.e. $b=b^{+}-b^{-}$, with $b^{+}_j,\ b^{-}_j\ge 0$.
	
	For a given support $S\subset\{1,\ldots,n\}$, we denote the range of $X_S$ by $V_S$ and the orthogonal projection onto $V_S$ by $\gP_{V_S}$. Recall that
	\bean
	\gP_{V_S}=X_S (X_S^tX_{S})^{-1} X_S^t.
	\eean
	
	The support of $\ch{\beta}_\lb$ is denoted by $\ch{T}_\lb$.
	For the sake of notational simplicity, we write
	\bea
	\ch{\beta}_{\ch{T}_\lambda} & := & \left(\ch{\beta}_{\lambda}\right)_{\ch{T}_\lambda}.
	\eea

	\section{Optimality conditions} \label{opt}
	
	In this section, we review the standard optimality conditions for the LASSO estimator. 
	A necessary and sufficient optimality condition in (\ref{lasso}) is that
	\bea
	0\in \partial \Big(\frac12 \|y-X \ch{\beta}_\lb\|_2^2+ \lambda 
	\|\ch{\beta}_\lb\|_1\Big),
	\eea
	where $\partial$ denotes the sub-differential, which is equivalent to the existence of $g_\lb$ in $\partial \|\cdot\|_1$ at
	$\ch{\beta}_\lb$ such that
	\bea
	\label{glamb}
	-X^t(y-X \ch{\beta}_\lb)+ \lambda g_\lb & = & 0.
	\eea
	On the other hand, the sub-differential of $\|\cdot\|_1$ at $\ch{\beta}_\lb$ is defined by
	\bean
	\partial \|\cdot\|_1(\ch{\beta}_\lb)=
	\left\{\gamma \in \R^p,\ \gamma_{\ch{T}_\lb}=\sgn(\ch{\beta}_{\ch{T}_\lb}) 
	\textrm{ and }\|\gamma_{\ch{T}_\lb^c}\|_{\infty}<1 \right\}.
	\eean
	Thus, using the fact that $y=X\beta+z$, we may easily conclude that 
	a necessary and sufficient condition for optimality in (\ref{lasso}) is the existence of a vector
	$g_\lb$, satisfying $g_{\ch{T}_\lb}=\sgn(\ch{\beta}_{\ch{T}_\lb})$ and
	$\|g_{\ch{T}_\lb^c}\|_{\infty}<1$, and such that
	\bea \label{opt1}
	X^t(y-X\ch{\beta}_\lb) & = & \lambda \ g_\lb. 
	\eea
	
	The following corollary is a direct but important consequence of these previous preliminary remarks.
	\begin{cor}
		\label{kk}
		A necessary and sufficient condition for a given random vector $\mathsf{b}$ with support 
		$\mathsf{T}$ to simultaneously satisfy the two following conditions:
		\begin{enumerate}
			\item $\mathsf{b}=\ch{\beta}_\lb$,
			\item $\mathsf{b}$ has the same support $T$ and sign pattern $\sgn(\beta_T)$ as $\beta$
		\end{enumerate} 
		is that 
		\bea
		\label{cond1}
		X_T^t (y-X\mathsf{b}) & = & \lb \ \sgn(\beta_T) \\
		\label{cond2}
		\|X_{T^c}^t (y-X\mathsf{b})\|_{\infty} & < & \lb.
		\eea
	\end{cor}
	
	\begin{proof}
		The fact that (\ref{cond1}) and (\ref{cond2}) are necessary is a straightforward 
		consequence of (\ref{opt1}). Conversely, assume that (\ref{cond1}) and (\ref{cond2}) hold. Set 
		\bea
		\label{glambprf}
		\mathsf{g} & = & \frac1{\lambda} X^t(y-X \mathsf{b}).
		\eea
		Using (\ref{glamb}), we deduce that $\mathsf{g}$ belongs to $\partial \|\cdot\|_1(\mathsf{b})$
		and that the support of $\mathsf{b}$ is exactly the set $\mathsf{T}=\{j\in \{1,\ldots,p\},\
		|\mathsf{g}_j|=1 \}$. On the other hand, we have that 
		\bea
		\mathsf{g} & = & \sgn(\beta_T) \\
		\|\mathsf{g}\|_{\infty} & < & 1,
		\eea
		and we may deduce that $\mathsf{g}$ is at the same time in the sub-differential of any vector 
		$b$ in $\mathbb R^p$ with same support and sign pattern as $\beta$. Therefore, we have   
		\bea 
		T & = \{j\in \{1,\ldots,p\},\
		|\mathsf{g}_j|=1 \} & = \mathsf{T},
		\eea
		and we conclude that $\beta$ and $\mathsf{b}$ have the same support. Moreover, the index set $T^+$ of the 
		positive components of $\beta$ and the index set $\mathsf{T}^+$ of the positive components of $\mathsf{b}$ 
		satisfy  
		\bea 
		T^+ & = \{j\in \{1,\ldots,p\},\ 
		\mathsf{g}_j=1 \} & = \mathsf{T}^+.
		\eea
		The same argument implies that the index set $T^-$ of the negative components of $\beta$ equals the index
		set $\mathsf{T}^-$ of the negative components of $\mathsf{b}$. To sum up, $\beta$ and $\mathsf{b}$ have the 
		same support and sign pattern and the proof is completed. This moreover implies that 
		(\ref{cond1}) and (\ref{cond2}) are the optimality conditions for (\ref{lasso}) and we obtain that 
		$\mathsf{b}=\ch{\beta}$ as announced. 
	\end{proof}
	
	\section{Uniqueness of the LASSO estimator}
	\label{lassolamb}
	
	\subsection{The General Position Condition}
	
	Our main assumption on the design matrix $X$ is the following. 
	\begin{ass}{\rm (General Position Condition for $X$)} 
		\label{ass4}
		For all supports $S\neq S^\prime\subset\{1,\ldots,n\}$ and all $(\ep_S, \ep_{S^\prime}) \in \{-1,1\}^{|S|}\times \{-1,1\}^{|S^\prime|}$ such that $X_{S}$ and $X_{S^\prime}$ are non-singular, we have
		\bea
		\label{gp}
		\ep_S (X_S^tX_{S})^{-1} \ep_S & \neq & \ep_{S^\prime}^t  (X_{S^\prime}^tX_{S^\prime})^{-1} \ep_{S^\prime} \\
		\ep_S (X_S^tX_{S})^{-1} (X_{S}^tX_{S^\prime}) (X_{S^\prime}^tX_{S^\prime})^{-1} \ep_{S^\prime}
		& \neq & \left| \ep_S (X_S^tX_{S})^{-1} \ep_S \right|.
		\eea
	\end{ass}
	Since $S\neq S^\prime$, this property clearly holds with probability one if the entries of $X$ are independent and have an absolutely continuous 
	density with respect to the Lebesgue measure. This is a generic situation in statistics where the covariate measurements 
	are usually corrupted by some noise. In the case of a more general type of design, we believe that this definition could easily 
	be generalized so as to guarantee that (\ref{gp}) fails with probability at most of the order $p^{-\alpha}$ or 
	is automatically satisfied for a carefully chosen deterministic design. A similar property, called General Position (GP) 
	was proposed in \cite{Dossal:CRAS12} for the problem of finding the sparsest solution of a linear system with application to Basis Pursuit. 
	
	This section establishes various continuity and monotonicity properties
	of some important functions of $\ch{\beta}_\lb$ using the General Position Condition assumption only. 
	
	The following notations will be useful. Define $\mathcal  L$ as the cost function:
	\bef{\mathcal  L}
	\R_{+}^{\star}\times \R^{p} &  \longrightarrow & \R_{+}\\
	(\lb,b) & \longmapsto & \d \frac12 \|y-X b\|_2^2+\lb \|b\|_1,
	\eef
	and for all $\lb>0$,
	\bea
	\theta(\lambda) & = & \inf_{b\in \mathbb R^p} \mathcal  L(\lb,b).
	\eea

	We begin with the following useful characterization of the LASSO estimators. For any $w\in\R^p$, let us introduce 
	\bea \label{grosp}
	\cal P(w) &= & \down{b\in \R^p,\ Xb=Xw}{\rm argmin} \|b\|_1.
	\eea
	
	\begin{lemm}
		\label{lprep} 
		A vector $\ch{\beta}_\lb$ is a solution of (\ref{lasso}) if and only if 
		$\ch{\beta}_\lb \in \cal P(\ch{\beta}_\lb)$.
	\end{lemm}
	
	\begin{proof} 
		Let $\ch{\beta}_\lb$ be a solution of (\ref{lasso}). Let $\wdt{\beta}_\lb \in \cal P(\ch{\beta}_\lb)$. Then, we have 
		\bea 
		\label{optl1}
		\|\tilde{\beta}_\lb\|_1 & \le & \|\ch{\beta}_\lb\|_1.
		\eea
		On the other hand, the definition of $\ch{\beta}_\lb$ implies that
		\bea 
		\label{optemp}
		\frac12 \|y-X\tilde{\beta}_\lb\|_2^2+\lb \|\tilde{\beta}_\lb\|_1 & \ge &  
		\frac12 \|y-X\ch{\beta}_\lb\|_2^2+\lb \|\ch{\beta}_\lb\|_1.
		\eea
		Moreover, since $X\tilde{\beta}_\lb=X\ch{\beta}_\lambda$, we have that 
		\bea 
		\label{l2fid}
		\frac12 \|y-X\tilde{\beta}_\lb\|_2^2 & = & \frac12 \|y-X\ch{\beta}_\lb\|_2^2,
		\eea 
		and subtracting this equality to (\ref{optemp}), we obtain that  
		\bean 
		\|\ch{\beta}_\lb\|_1  & \le & \|\tilde{\beta}_\lb\|_1,
		\eean
		which, combined with (\ref{optl1}), implies that 
		\bea
		\label{1pen} 
		\|\ch{\beta}_\lb\|_1  & = & \|\tilde{\beta}_\lb\|_1.
		\eea
		This last equality together with (\ref{l2fid}) implies the desired result.
	\end{proof}
	
	We now give a useful expression of 
	$\ch{\beta}_\lambda$ in terms of $\lb$ and the submatrix of $X$ indexed by $\ch{T}$. 
	\begin{lemm}
		\label{} 
		For any $\lambda>0$ such that $\ch{\beta}_\lambda\neq 0$, the matrix $X_{\ch{T}_\lb}$ is non-singular and we have
		\bea
		\label{turtle}
		\ch{\beta}_{\ch{T}_{\lb}} & = & (X_{\ch{T}_\lb}^tX_{\ch{T}_\lb})^{-1}\left(X_{\ch{T}_\lb}^ty-\lambda \ \sgn\left(\ch{\beta}_{\ch{T}_\lb}\right)\right).
		\eea
	\end{lemm}
	
	\begin{proof}  
		Recall that the optimality conditions for the LASSO imply that 
		\bea 
		X_{\ch{T}_\lb}^t(y-X_{\ch{T}_\lb}\ch{\beta}_{\ch{T}_\lb}) & = & \lb \ \sgn(\ch{\beta}_{\ch{T}_\lb}).
		\eea 
		Since $X_{\ch{T}_\lb}$ is non-singular, we obtain (\ref{turtle}).
	\end{proof}  
	
	The following Theorem establishes the unicity of $\ch{\beta}_\lb$ and shows that its support is of size at most $n$.
	\begin{theo}
		\label{uniq} 
		Assume that Assumption \ref{ass4} holds. Then, almost surely, for any $\lambda>0$, the minimization problem (\ref{lasso}) 
		has a unique solution $\ch{\beta}_\lambda$, and its support $\ch{T}_{\lb}\subset \{1,\ldots,p\}$ verifies
		\bea
		|\ch{T}_\lambda| & \le & n.
		\eea 
	\end{theo}
	
	\begin{proof} 
		We first study the support of a possible solution $\ch{\beta}_\lb$. Second, we derive (\ref{turtle}), and eventually, we prove the uniqueness of $\ch{\beta}_\lb$ under the general position condition.
		
		\subsubsection*{Study of $\# \ch{T}$} Recall that $b^{+}$ (resp. $b^{-}$) be the non-negative (resp. non-positive)
		part of $b$, i.e. $b=b^{+}-b^{-}$, with $b^{+}_j,\ b^{-}_j\ge 0$.
		Then, Lemma \ref{lprep} above equivalently says that $\ch{\beta}_\lb$ is a solution 
		of (\ref{lasso}) if and only if $\ch{\beta}_\lambda^{+}$ and $\ch{\beta}_\lambda^{-}$ are
		solutions of
		\begin{equation}
		\label{linp}
		\min_{b^+,\ b^- \in \R_+^p}
		\sum_{j=1}^{p} \left\{b_j^{+}+b_j^{-}\right\} \ \textrm{s.t.}\ Xb^+-Xb^-=X\ch{\beta}_\lambda.
		\end{equation}
		The remainder of the proof relies on linear programming theory and
		Assumption \ref{ass4}. Notice first that the solution set is compact due to
		the coercivity of the $\ell_1$-norm. Thus, the theory of linear
		programming \cite{Schrijver:ThLIP86} ensures that each extreme point of
		the solution set of (\ref{linp}) is completely determined by a "basis"
		$B$. In the present setting, for an extreme point $b^*={b^*}^+-{b^*}^-$ of the
		solution set of (\ref{linp}), the associated basis $B^*$ can be written (in a non-unique way)
		as $B^*=B^{*+}\cup B^{*-}$, $|B^*|=n$, and is such that
		\begin{itemize}
			\item[(i)] the square matrix $[X_{B^{*+}},-X_{B^{*-}}]$ is non singular,
			\item[(ii)]  ${b^*}_{{B^{*}}^c}=0$ and 
			\item[(iii)]  the couple $({b^*}^+_{B^{*+}},{b^*}^-_{B^{*-}})$ is uniquely determined by the system
			\bea 
			X_{B^{*+}}{b^*}^+_{B^{*+}}-X_{B^{*-}} {b^*}^-_{B^{*-}} & = &
			X\ch{\beta}_\lambda, 
			\eea 
			(or equivalently, $X_{B^*}b_{B^*}^* = X \ch{\beta}_\lambda$). 
		\end{itemize}
		An immediate consequence is that the support of $b^*$ has cardinal at most $n$. Moreover, 
		$b^*\in \cal P(b^*)$,
		and using Lemma \ref{lprep}, we deduce that $b^*$ is a solution of (\ref{lasso}).
		Therefore, we may assume without loss of generality that $\ch{\beta}_\lb$ is an extreme point 
		of $\cal P(\ch{\beta}_\lb)$, with
		$$\# \ch{T}_\lb \le n$$
		and that $X_{\ch{T}_\lb}$ is non-singular.

		\subsubsection*{Uniqueness of $\ch{\beta}_\lambda$: first part} --- We 
		give two equations satisfied by $\lb$ and $z$ in the case where uniqueness of 
		the LASSO estimator fails. 
		
		Let $\ch{\beta}_\lb^\prime$ in $\mathbb R^p$ be another 
		solution of (\ref{lasso}). Using the same reasonning as for $\ch{\beta}_\lb$ in the end of the last paragraph, 
		we may assume w.l.o.g. that the support $\ch{T}_\lb^\prime$ of $\ch{\beta}_\lb^\prime$ has cardinal at most $n$
		and that $X_{\ch{T}_\lb^\prime}$ is non-singular. 
		Convexity of the LASSO functional implies that the map
		\bef{\phi}
		\label{couac}
		[0,1] & \longrightarrow & \R_+\\
		t & \longmapsto & \cal L \left(\lb, (t\ \ch{\beta}_\lb +(1-t)\ \ch{\beta}_\lb^\prime)\right)
		\eef
		is constant.
		
		Notice that the term $\|\ch{\beta}_\lb^\prime + t \ (\ch{\beta}_\lb - \ch{\beta}_\lb^\prime)\|_1$ is 
		in fact piecewise affine on $(0,t)$. 
		Set
		\bean 
		\rho_\lb & = & \sgn(\ch{\beta}_{\ch{T}_\lb}) \\
		\rho_\lb^\prime & = & \sgn(\ch{\beta}_{\ch{T}_\lb^\prime}^\prime).
		\eean 
		Now, let $t^\star>0$ sufficiently small such that for all $t\in(0,t^\star)$ the support of 
		$\ch{\beta}_\lb^\prime + t \ (\ch{\beta}_\lb - \ch{\beta}_\lb^\prime)$ is constant and equal to 
		$\ch{T}_\lb \cup \ch{T}_\lb^\prime$ and no sign change occurs. 
		Set 
		\bea 
		\rho & = & \sgn \left( \left( \ch{\beta}_\lb^\prime + t \ (\ch{\beta}_\lb - \ch{\beta}_\lb^\prime)\right)
		_{\ch{T}_\lb \cup \ch{T}_\lb^\prime}\right).
		\eea
		Thus, for all $t\in(0,t^*)$,   
		\bean 
		\|\ch{\beta}_\lb^\prime + t \ (\ch{\beta}_\lb - \ch{\beta}_\lb^\prime)\|_1
		& = & \rho^t \ch{\beta}_\lb^\prime + t \ \rho^t (\ch{\beta}_\lb - \ch{\beta}_\lb^\prime)
		\eean
		with
		\bean
		\rho_{\ch{T}_\lb}  = \rho_\lb \hspace{.3cm} & \textrm{ and } \hspace{.3cm} 
		\rho_{\ch{T}_\lb^\prime}  & =  \rho_\lb^\prime 
		\eean
		and we deduce that $\phi$ is a second order polynomial in the variable $t\in(0,t^*)$. 
		Therefore, the coefficients corresponding to the quadratic and linear terms of $\phi$ must be zero. 
		Developing the term $\frac12 \|y-X (t\ \ch{\beta}_\lb +(1-t)\ \ch{\beta}_\lb^\prime)\|_2^2$, we then obtain:
		\bean
		X_{\ch{T}_\lb}\ch{\beta}_\lb - X_{\ch{T}_\lb^\prime}\ch{\beta}_\lb^\prime & = & 0\\
		y^t (X_{\ch{T}_\lb}\ch{\beta}_\lb - X_{\ch{T}_\lb^\prime}\ch{\beta}_\lb^\prime)
		+ \lb \ \rho^t (\ch{\beta}_\lb - \ch{\beta}_\lb^\prime) & = & 0,
		\eean
		which is equivalent to 
		\bea 
		\label{quadoo}
		X_{\ch{T}_\lb}\ch{\beta}_\lb - X_{\ch{T}_\lb^\prime}\ch{\beta}_\lb^\prime & = & 0\\
		\label{bleugle}
		\rho^t (\ch{\beta}_\lb - \ch{\beta}_\lb^\prime) & = & 0.
		\eea

		\subsubsection*{Uniqueness of $\ch{\beta}_\lambda$: second part} --- As for $\ch{\beta}_{\ch{T}_\lb}$, we write
		\bea
		\label{turtleprime}
		\ch{\beta}_{\ch{T}_\lb^\prime}^\prime & = & 
		(X_{\ch{T}_\lb^\prime}^tX_{\ch{T}_\lb^\prime})^{-1}  \left( X_{\ch{T}_\lb^\prime}^t y
		- \lb \ \sgn(\ch{\beta}_{\ch{T}_\lb^\prime})\right).
		\eea
		Replacing (\ref{turtle}) and (\ref{turtleprime}) into (\ref{quadoo}), we 
		obtain
		\beq
		\label{quadoo2}
		(\gP_{\ch{T}_\lb}-\gP_{\ch{T}_\lb^\prime}) \ y
		-\lb \left(X_{\ch{T}_\lb} (X_{\ch{T}_\lb}^tX_{\ch{T}_\lb})^{-1} \rho_\lb -X_{\ch{T}_\lb^\prime} (X_{\ch{T}_\lb^\prime}^tX_{\ch{T}_\lb^\prime})^{-1} \rho_\lb^\prime\right)  =  0.
		\eeq
		On the other hand, (\ref{bleugle}) gives 
		\bea
		\label{beugle2}
		0 & = & 
		\ y^t \left(X_{\ch{T}_\lb} (X_{\ch{T}_\lb}^tX_{\ch{T}_\lb})^{-1} \rho_\lb
		- X_{\ch{T}_\lb^\prime}  (X_{\ch{T}_\lb^\prime}^tX_{\ch{T}_\lb^\prime})^{-1} \rho_\lb^\prime  \right) \\
		& & -\lb \left(\rho_{\lb}^t (X_{\ch{T}_\lb}^tX_{\ch{T}_\lb})^{-1}  \rho_\lb -(\rho_\lb^\prime)^t (X_{\ch{T}_\lb^\prime}^tX_{\ch{T}_\lb^\prime})^{-1}  \rho_\lb^\prime \right).\nonumber
		\eea
		Setting
		\bean
		\eta_\lb & = & X_{\ch{T}_\lb} (X_{\ch{T}_\lb}^tX_{\ch{T}_\lb})^{-1} \rho_\lb
		- X_{\ch{T}_\lb^\prime}  (X_{\ch{T}_\lb^\prime}^tX_{\ch{T}_\lb^\prime})^{-1} \rho_\lb^\prime \\
		\zeta_\lb & = & \rho_{\lb}^t (X_{\ch{T}_\lb}^tX_{\ch{T}_\lb})^{-1}  \rho_\lb -(\rho_\lb^\prime)^t (X_{\ch{T}_\lb^\prime}^tX_{\ch{T}_\lb^\prime})^{-1}  \rho_\lb^\prime,
		\eean
		we obtain the system:
		\bea 
		(\gP_{\ch{T}_\lb}-\gP_{\ch{T}_\lb^\prime}) y
		-\lb \eta_\lb & = & 0 \\
		y^t \eta_\lb - \lb \zeta_\lb & = & 0.
		\eea
		Notice that
		\bean
		\left( \gP_{\ch{T}_\lb}-\gP_{\ch{T}_\lb^\prime},\ \eta_\lb ,\ \zeta_\lb \right) 
		& \in & \cal F_1 \times \cal F_2\times \cal F_3,
		\eean
		where
		\bean
		\cal F_1& = & \left\{\gP_{S}-\gP_{S^\prime}, \ S\neq S^\prime\subset\{1,\ldots,n\} \right\} \\
		\cal F_2 & = & \left\{  X_S (X_S^tX_{S})^{-1} \ep_S
		- X_{S^\prime}  (X_{S^\prime}^tX_{S^\prime})^{-1} \ep_{S^\prime},\  (S,S^\prime,\ep_S,\ep_{S^\prime})\in \cal G \right\} \\
		\cal F_3 & = & \left\{ \ep_S^t (X_S^tX_{S})^{-1}  \ep_S - \ep_{S^\prime}^t (X_{S^\prime}^tX_{S^\prime})^{-1}  \ep_{S^\prime}    \  (S,S^\prime,\ep_S,\ep_{S^\prime})\in \cal G \right\},
		\eean
		with 
		\bean
		\cal G & = & \left\{ S\neq S^\prime\subset\{1,\ldots,n\},\ (\ep_S, \ep_{S^\prime}) \in \{-1,1\}^{|S|}\times \{-1,1\}^{|S^\prime|} \right\} .
		\eean
		Therefore, $(y,\lb)$ is a solution of the finite set of equations
		\bea 
		Q \ y -\lb \eta & = & 0 \label{sysquad12} \\
		y^t \eta - \lb \zeta & = & 0, \label{sysquad22}
		\eea
		when $(Q,\eta,\zeta)$ is running over $\cal F_1 \times \cal F_2\times \cal F_3$. This implies that 
		\bean
		\left\{ \left(\ch{\beta}_\lb,\lb \right), \ \lb>0 \right\} & \subset &  \bigcup_{j\in\cal J}E_j,
		\eean
		where $\cal J$ is a finite set and the $E_j\subset \R^{n+1}$ are linear subspaces. 
		
		Let us now show that there is no $E_j$, $j\in\cal J$, containing a subspace of dimension $n$. Let us suppose that this is not the case, i.e. there exist two supports $S\neq S^\prime$ and $(\eta,\zeta)\in \cal F_2\times \cal F_3$ such that for all $y\in\R^n$,
		\bea \label{kerncoroll}
		(\gP_S-\gP_{S^\prime}) y & = & \frac{\eta \eta^t}{\zeta}\ y.
		\eea
		When the rank of $\gP_S-\gP_{S^\prime}$ is different from $1$, (\ref{kerncoroll}) cannot be satisfied for all $y\in\R^n$. 
		Thus, we only have to focus on the case where the rank of $\gP_S-\gP_{S^\prime}$ is $1$, 
		or equivalently, $|S\Delta S^\prime|=1$. We distinguish two cases. 
		Either $W_{S}:=V_{S}^{\perp}\cap V_{S^{\prime}}\neq \{0\}$ or $W_{S}= \{0\}$:
		\begin{itemize}
			\item[(i)] If $W_{S}\neq \{0\}$, take $v\in W_{S}$, $v\neq0$. Then $(\gP_S-\gP_{S^\prime})v=-v$, and the only eigenvalue of $\gP_S-\gP_{S^\prime} $ is $-1$.
			\item[(ii)] If $W_{S}= \{0\}$, then $V_{S^{\prime}} \subset V_{S}$ and so $W_{S^{\prime}}:=V_{S^\prime}^{\perp}\cap V_{S}\neq \{0\}$. Hence, take a non-zero $v\in W_{S^{\prime}}$. We now have $(\gP_S-\gP_{S^\prime})v=v$, and the only eigenvalue of $\gP_S-\gP_{S^\prime} $ is $1$.
		\end{itemize}
		But the only eigenvalue of $\eta \eta^t /\zeta$ is $\|\eta\|^2_2/\zeta$.
		By developing 
		\bean
		\|\eta\|_2^2 & = & \|X_S (X_S^tX_{S})^{-1} \ep_S
		- X_{S^\prime}^t  (X_{S^\prime}^tX_{S^\prime})^{-1} \ep_{S^\prime}\|^{2}  
		\eean
		and comparing with
		\bean
		\zeta & = & \ep_S (X_S^tX_{S})^{-1} \ep_S
		- \ep_{S^\prime}^t  (X_{S^\prime}^tX_{S^\prime})^{-1} \ep_{S^\prime},
		\eean
		we can write that the General Position Condition, Assumption \ref{ass4}, is equivalent to the following inequations:
		\bean 
		\zeta & \neq & 0 \\
		\|\eta\|_2^2 & \neq & |\zeta|.
		\eean  
		Therefore, the operators $\gP_S-\gP_{S^\prime}$ and $\eta \eta^t /\zeta$ are different. Hence, (\ref{kerncoroll})
		is not satisfied for all $y\in \R^n$ when the rank of $\gP_S-\gP_{S^\prime}$ is $1$. \\
		
		As a conclusion, the dimension of $E_j$ is less than $n+1$. the probability that there exists $\lb>0$ such that 
		uniqueness of the LASSO estimator fails, is equal to zero.   
	\end{proof}

	\bibliographystyle{amsplain}
	\bibliography{database}

\end{document}